\documentclass{article}
\usepackage{arxiv}

\usepackage[utf8]{inputenc} % allow utf-8 input
\usepackage[T1]{fontenc}    % use 8-bit T1 fonts
\usepackage{hyperref}       % hyperlinks
\usepackage{url}            % simple URL typese}ing
\usepackage{booktabs}       % professional-quality tables
\usepackage{amsfonts}       % blackboard math symbols
\usepackage{nicefrac}       % compact symbols for 1/2, etc.
\usepackage[english]{babel}
\usepackage{ae}
\usepackage{eucal}
\usepackage[dvips]{graphicx}
\usepackage{epsfig}
\usepackage{graphicx}
\usepackage{amsfonts,amsthm}
\usepackage{amssymb}
\usepackage{a4}
\usepackage{apu}
\usepackage{xcolor}
\usepackage{amsmath}
\usepackage{mathtools}
\usepackage{multicol}

\usepackage{dcolumn,amsthm}

\renewenvironment{proof}[1][Proof]{\noindent\textit{#1. } }{\hfill$\square$}

\newtheoremstyle{theorem}{6pt}{6pt}{\rm}{}{\sffamily}{ }{ }{}
\theoremstyle{theorem}

\newtheoremstyle{lemma}{6pt}{6pt}{\rm}{}{\sffamily}{ }{ }{}
\theoremstyle{lemma}

\newtheoremstyle{example}{6pt}{6pt}{\rm}{}{\sffamily}{ }{ }{}
\theoremstyle{example}

\newtheoremstyle{corollary}{6pt}{6pt}{\rm}{}{\sffamily}{ }{ }{}
\theoremstyle{corollary}

\newtheoremstyle{definition}{6pt}{6pt}{\rm}{}{\sffamily}{ }{ }{}
\theoremstyle{definition}

\newtheoremstyle{remark}{6pt}{6pt}{\rm}{}{\sffamily}{ }{ }{}
\theoremstyle{remark}

\newtheoremstyle{approximation}{6pt}{6pt}{\rm}{}{\sffamily}{ }{ }{}
\theoremstyle{approximation}

\newtheoremstyle{scheme}{6pt}{6pt}{\rm}{}{\sffamily}{ }{ }{}
\theoremstyle{scheme}

\usepackage{csquotes}

\title{Explicit solutions to the 3D incompressible Euler equations in Lagrangian formulation}

\author{
   Tomi Saleva \\
   Department of Physics and Mathematics\\
   University of Eastern Finland\\
   P.O. box 111, FI-80101 Joensuu, Finland  \\
   \text{tomisal@student.uef.fi} \\
      \And
   Jukka Tuomela \\
   Department of Physics and Mathematics\\
   University of Eastern Finland\\
   P.O. box 111, FI-80101 Joensuu, Finland  \\
   \text{jukka.tuomela@uef.fi} \\
}

\begin{document}

\maketitle

\begin{abstract}
We introduce many families of explicit solutions to the three dimensional incompressible Euler equations for nonviscous fluid flows using the Lagrangian framework. Almost no exact Lagrangian solutions exist in the literature prior to this study. We search for solutions where the time component and the spatial component are separated, applying the same ideas we used previously in the two dimensional case. We show a general method to derive separate constraint equations for the spatial component and the time component. Using this provides us with a plenty of solution sets exhibiting several different types of fluid behaviour, but since they are computationally heavy to analyze, we have to restrict deeper analysis to the most interesting cases only. It is also possible and perhaps even probable that there exist more solutions of the separation of variables type beyond what we have found.
\end{abstract}

\textbf{\emph Mathematics Subject Classification (2020)} 35A09; 35Q31; 76B99

\keywords{Explicit solutions, Euler equations, Lagrangian formulation, Fluid mechanics}

\thanks{The first author was supported by Finnish Cultural Foundation.}

\section{Introduction}

This is the third paper, following \cite{maju,toju}, in our research of finding explicit solutions to the incompressible Euler equations in the Lagrangian framework. The Lagrangian framework is one of the two main approaches to describing fluid motion. In the Lagrangian representation the flow is described by giving the trajectories of individual particles. The other main approach is the Eulerian framework which considers the velocity field or the vorticity field of the flow. For general background on the Euler equations we refer to \cite{constantin}, and for the Lagrangian formulation in particular we refer to \cite{bennett}. Lagrangian formulation has been applied to flow problems with group theory in \cite{andreev}.   

In our two previous articles we studied the two-dimensional case, finding several solution sets, and in the present paper we apply similar ideas to the three-dimensional case. The goal is to find solutions of the separation of variables type, where space and time coordinates are treated separately.  The first explicit solutions of this type were already found in 19th century by Gerstner and Kirchhoff in the 2D case \cite{G,K}. In the recent decades these solutions have been generalized using harmonic maps, see for example \cite{Abrashkin,AY,AC,CM}. All these solutions were of the form where the time component and the spatial component could be separated, and finally in \cite{toju} we systematically found probably all possible solutions of the separation of variables type to the 2D Euler equations in Lagrangian formulation, including ones that could not be found with harmonic maps.

However, it seems that there are surprisingly few articles dealing with the 3D case, the only exception we found is \cite{yakzen}. After all, many scientists have studied this problem for a long time in the two-dimensional context and it would seem natural to try to find similar solutions also in three dimensions which is of course physically the most relevant case. Perhaps one reason is that since classically complex functions were used to analyze the two-dimensional case, the generalization to the three-dimensional case appeared difficult. 

In \cite{yakzen} the problem is approached in a slightly nonconventional variant of the Lagrangian formulation, not describing the system for the particle trajectory map $\fii$ but rather for its differential $d\fii$. For the spatial part they in fact used the complex formulation for two of the components and a real function for the third component. To find exact solutions, they considered cases where the time component and the spatial component are separated for $d\fii$. A simple integration argument shows that then $\fii$ is also of the separation of variables form and thus these solutions can also be found with the technique presented in this paper.

Anyway in \cite{maju, toju} we showed that even in two dimensions the situation is best analyzed using real functions, and thus a priori it seemed natural to believe that similar techniques should be applicable also  in three dimensional case. In the present paper we show that this is indeed the case: even though certain classes of solutions can equivalently be analyzed with complex formulation, one can find many more large classes of solutions operating with real functions. In fact it is not possible to describe all solutions: there are simply so many possibilities that one cannot treat them all in a short article like this. Also it is not possible to describe all computations in detail. Some intermediate formulas and expressions are simply too big to be written down. However, this is not a problem in the sense that the verification of the final results is a straightforward computation. Of course this computation would be very tedious if done by hand. 

Perhaps one should also clarify what we  mean by the word ''explicit'' in the title. In some cases the solution set cannot really be described by completely explicit formulas. In these cases one may think that the word ''explicit'' refers to the explicit description of the structure of the solutions. However, we think that these structural descriptions can also be very useful in applications: combining numerical and symbolic computation one can use them to ''build'' specific solutions with the desired properties. Also one can then analyze what kind of properties are actually possible for a given family of solutions.

In Section 2 we introduce some useful notions and results which are needed later. In Section 3 we formulate our framework precisely. Then in Sections 4 to 8 we present our solutions. We start with simple cases which can be described more thoroughly; these cases show the reader the main ideas in the computations. In later sections many details must be suppressed, but the computations are really very similar to simpler cases. One does not need any new methods to deal with the complicated cases, the difficulty comes just from the sheer size of the problem. Unlike in the 2D case we studied in \cite{toju}, here we will not consider all possible families of solutions due to their vast number. Instead, we concentrate on those cases which to us seem to be the most promising from the physical point of view. However, we also indicate some systems which we know lead to nontrivial explicit solutions but which are not analyzed in the present paper. It is also highly possible that there are cases that we did not discover altogether.

\section{Preliminaries and notation}

\subsection{Euler parameters}

\label{euler-parameters}
We need to represent rotations in $\mathbb{R}^3$. There are several ways to do this but  Euler parameters are most convenient for our purposes.
Let $a=(a_{0},a_{1},a_{2},a_{3})$ and let us define
the following matrices:
\begin{align*}
\widetilde{H}_a& =
\begin{pmatrix}
-a_{1} & a_{0} & -a_{3} & a_{2} \\
-a_{2} & a_{3} & a_{0} & -a_{1} \\
-a_{3} & -a_{2} & a_{1} & a_{0}
\end{pmatrix}
\quad ,\quad H_a=
\begin{pmatrix}
-a_{1} & a_{0} & a_{3} & -a_{2} \\
-a_{2} & -a_{3} & a_{0} & a_{1} \\
-a_{3} & a_{2} & -a_{1} & a_{0}
\end{pmatrix}
\\[3mm]
\tilde K_a &=
\begin{pmatrix}
a_0 & a_1 & a_2 & a_3 \\
-a_{1} & a_{0} & -a_{3} & a_{2} \\
-a_{2} & a_{3} & a_{0} & -a_{1} \\
-a_{3} & -a_{2} & a_{1} & a_{0}
\end{pmatrix}
\quad ,\quad K_a=
\begin{pmatrix}
a_0 & a_1 & a_2 & a_3 \\
-a_{1} & a_{0} & a_{3} & -a_{2} \\
-a_{2} & -a_{3} & a_{0} & a_{1} \\
-a_{3} & a_{2} & -a_{1} & a_{0}
\end{pmatrix}
\\[3mm]
R_a=\widetilde{H}_aH^{T}_a & =
\begin{pmatrix}
a_0^2+a_1^2- a_2^2-a_3^2 & 2a_1a_2-2a_0a_3 &
2a_1a_3+2a_0a_2 \\
2a_1a_2+2a_0a_3 & a_0^2-a_1^2+a_2^2-a_3^2 &
2a_2a_3-2a_0a_1 \\
2a_1a_3-2a_0a_2 & 2a_2a_3+2a_0a_1 & a_0^2-a_1^2-
a_2^2+a_3^2
\end{pmatrix}
\ .
\end{align*}
If $|a|=1$ then $R_a\in\mathbb{SO}(3)$ and $\tilde K_a$, $K_a\in
\mathbb{SO}(4)$. Since  $R_a=R_{-a}$ the sphere $S^3$ is the double cover of $\mathbb{SO}(3)$.

\subsection{PDE}
Let $u$ and $v$ be functions of two variables; there are two natural conventions regarding the sign in the Cauchy--Riemann equations:
\[
  \begin{cases}
      u_x-v_y=0\\
      u_y+v_x=0
  \end{cases}\quad,\quad
  \begin{cases}
      u_x+v_y=0\\
      u_y-v_x=0
  \end{cases}\ .
\]
It turns out that for us it is more convenient to use the second system. We will call this the anti CR system. 

For general PDE systems it is more convenient to use multiindices to denote derivatives. Hence if $u=(u^1,\dots,u^m)$  is some map the derivatives of its components are
\[
    u_\nu^j=\partial^\nu u^j=
    \frac{\partial^{|\nu|}u^j}{\partial x_1^{\nu_1}\cdots\partial x_n^{\nu_n}}\ .
\]
We will say several times below that a solution to a certain PDE system exists by standard theorems.  This means that the necessary details can easily be found in \cite{evans}.
We will also need some elementary notions of differential geometry in the analysis; for details we refer to \cite{lee}.

We will analyze below many rather complicated ODE and PDE systems. Since the analysis often cannot be given very explicitly it is perhaps helpful to discuss the general nature of systems and their solutions.  For a general treatment of these issues we refer to \cite{werner} and the many references therein. 

Now all our equations are differential polynomials so that some kind of algorithmic treatment is possible at least in principle. A fundamental theorem in differential algebra says that
\begin{itemize}
    \item[] \emph{any radical differential ideal is a finite intersection of prime differential ideals.} 
\end{itemize}
One may interpret this by saying that prime differential ideals correspond to certain families of solutions, and hence the theorem says that there are in fact only finitely many essentially different families of solutions. 

Another complication is that our systems contain parameters, and in general one must make case distinctions, depending on the values of parameters. The differential elimination algorithm \textsf{rifsimp} \cite{rif0} which is implemented in {\sc Maple} and which we will use at times follows the  strategy that the default is to look for the largest family. However, it must be kept in mind that there may be also other solutions, not contained in the family given.  

In the solutions given below we must at times make case distinctions so that there are several families of solutions, i.e. the differential ideal is not prime. However, we do not need to take into account all cases because we can discard the cases which are irrelevant in the present context. We will explain below more precisely what this means.

\subsection{Minors}
Let $A\in \mathbb{R}^{3\times m}$; the columns of $A$ are denoted by $A_j$. Let us further set $p_{ijk}=\det(A_i,A_j,A_k)$. The following formulas will be useful in computations.
\begin{lemma}
\begin{align*}
 &   p_{123}p_{145}-p_{124}p_{135}+p_{125}p_{134}=0\, ,\\
  &  p_{123}p_{456}-p_{124}p_{356}+p_{125}p_{346}-p_{126}p_{345}=0\\
  &p_{234}A_1-p_{134}A_2+p_{124}A_3-p_{123}A_4=0\\
  &  p_{123}^2p_{456}+\det{\begin{pmatrix} p_{234}&p_{235}&p_{236}\\ p_{134}&p_{135}&p_{136} \\ p_{124}&p_{125}&p_{126} \end{pmatrix}}=0\, .
\end{align*} 
\label{plu}
\end{lemma}
The first two equations in  Lemma \ref{plu} are examples  of Pl\"ucker relations.

\section{Euler equations}
Let us consider the incompressible Euler equations
\begin{equation}
    \begin{aligned}
    u_t+u\nabla u+\nabla p=0\\
    \nabla\cdot u=0
\end{aligned}
\label{euler}
\end{equation}
in some domain $\Omega\subset\mathbb{R}^n$. This is called the Eulerian description of the flow and the coordinates of $\Omega$, denoted $x$, are the Eulerian coordinates. Below we will consider another description which is almost the Lagrangian description of the flow. 

Let $D\subset\mathbb{R}^n$ be another domain and let us consider a family of diffeomorphisms $\varphi^t\,:\,D\to \Omega_t=\varphi^t(D)$. The coordinates in $D$ are denoted by  $z$. We can also define 
\[
  \varphi\,:\,D\times \mathbb{R}\to \mathbb{R}^n\quad,\quad
  \varphi(z,t)=\varphi^t(z)\, .
\]
Now given such $\varphi$ we can define the associated vector field $u$ by the formula
\begin{equation}
\frac{\partial}{\partial t} \fii(z,t)=u(\fii(z,t),t)\,.
\label{siirto}    
\end{equation}
Our goal is to find maps $\varphi$ such that $u$ solves the Euler equations in the three dimensional case. 

To express the incompressibility condition $\nabla\cdot u=0$ we introduce the volume forms $\mathsf{vol}_x=dx^1\wedge dx^2\wedge dx^3$ and $\mathsf{vol}_z=dz^1\wedge dz^2\wedge dz^3$. Then $\varphi^\ast \mathsf{vol}_x=\alpha\,\mathsf{vol}_z$ for some $\alpha$. In the Lagrangian case we should require that $\alpha=1$ but here we only require that it is independent of time and nonzero in $D$.

 It will be convenient to interpret the first equation of \eqref{euler}  as well as the equation \eqref{siirto} in terms of covectors instead of vectors. Then differentiating the equation \eqref{siirto} we obtain
 \[
    \varphi''+dp=0\,.
 \]
Then the condition of existence of local solutions is simply $d \varphi''=0$. The exterior derivative cannot be directly computed in $x$ coordinates so we have to pull back this condition to $z$ coordinates. To this end let us introduce
\[
    h=d\, \varphi^\ast\,\varphi'=
    d\varphi_1'\wedge d\varphi_1+d\varphi_2'\wedge d\varphi_2+
  d\varphi_3'\wedge d\varphi_3\,.
\]
Note that if we interpret the vorticity $\zeta=du$ as a two-form then $h$ is the vorticity in $z$ coordinates. A fundamental fact is that denoting by $\mathcal{L}$ the Lie derivative we have \cite{besfri,taylor3}
\[
  \zeta_t+\mathcal{L}_u \zeta=0\,.
\]
We will often write for simplicity
\[
 h=h^3dz^1\wedge dz^2-h^2dz^1\wedge dz^3+
 h^1 dz^2\wedge dz^3=\big(h^1,h^2,h^3\big)\ .
\]
The components $h^j$ are classically known as \emph{Cauchy invariants} \cite{besfri}, and as the name suggests they should not depend on time. Accordingly we see that  $d \varphi''=0$ is equivalent to $h'=0$. 
 Let us summarize this discussion with
\begin{lemma}
 $\fii$ provides a solution to Euler equations if and only if both $h$ and $\alpha$ are independent of time.
 \label{kriteeri}
\end{lemma}
Note that here is an essential difference to the 2D case: the condition for $h$ is related to two-forms while the condition for $\alpha$ is formulated in terms of three-forms. In 2D case both conditions are evidently formulated with two-forms. 

It is also interesting to express the vorticity $\zeta$ as a vector in $x$ coordinates in terms of $h$. To this end we express first $h$ as a vector using the Hodge star and raising the index; in components this gives
\[
  \tilde h=\sharp \ \ast h=\tilde h^1\partial_{z_1}+\tilde h^2\partial_{z_2}+
  \tilde h^3\partial_{z_3}
  =\big(\tilde h^1,\tilde h^2,\tilde h^3\big)=
  \frac{1}{\alpha}\,\big(h^1,h^2,h^3\big)\ .
\]
From this we get \emph{Cauchy's vorticity formula}  \cite{besfri}:
\[
   \zeta=  d\varphi\,\tilde h\, .
\]
This can be interpreted either as a pushforward or how the components of vector field change under the change of coordinates.

Let us then consider the maps of the following  form
\begin{equation}
    \varphi(z,t)=A(t)v(z)\, ,
\label{yrite}
\end{equation}
where $A(t)\in\mathbb{R}^{3\times m}$, $v\,:\, D\to\mathbb{R}^m$ and $D\subset\mathbb{R}^3$ is some  coordinate domain. Our next task is to compute convenient formulas for $h$ and $\alpha$ in this case.

Since all the analysis is local, the precise nature of $D$ is not important in our context. We will try to find maps $\varphi$ such that the corresponding vector field $u$ defined by the formula \eqref{siirto} is a solution to the Euler equations. Hence we should find $A$ and $v$ such that the conditions in Lemma \ref{kriteeri} are satisfied.
 Since we want that $\det(d\varphi)\ne0$, this necessarily implies that $\mathsf{rank}(A)=\mathsf{rank}(dv)=3$. Of course we can always assume that $\det(d\varphi)>0$.

Let us denote the coordinates of the intermediate space $\mathbb{R}^m$ by $y$. Then we can interpret $A$ as a map $\mathbb{R}^m\to\mathbb{R}^3$ and we can write
\[
  A^\ast\mathsf{vol}_x=\sum _{1\leq i<j<k\leq m} p_{ijk}dy^i\wedge dy^j\wedge dy^k=
  \sum _{1\leq i<j<k\leq m} \det (A_i,A_j,A_k)dy^i\wedge dy^j\wedge dy^k\ ,
\]
where $A_i$ denote the columns of $A$. Similarly we set
\[
   v^\ast dy^i\wedge dy^j\wedge dy^k=
   g_{ijk} \mathsf{vol}_z=\det (\nabla v^i,\nabla v^j,\nabla v^k)\mathsf{vol}_z\ .
\]
But from this we immediately obtain 
\begin{equation}
      \alpha=\sum _{1\le  i<j<k\le m} p_{ijk}g_{ijk}\ .
\label{a-rajehto}
\end{equation}
This is known as the Cauchy--Binet formula. Then let us introduce
\[
   \beta= \sum _{1\leq i<j\leq m} Q_{ij} dy^i\wedge dy^j\quad
  \mathrm{where}\quad
  Q_{ij}=\langle A_i',A_j\rangle - \langle A_j',A_i\rangle\ .
\]
Then putting $G_{ij}=v^\ast dy^i\wedge dy^j$ we can write
\begin{equation}
  h
  =v^\ast\beta= \sum _{1\leq i<j\leq m} Q_{ij}G_{ij}
  = \sum _{1\leq i<j\leq m} Q_{ij} dv^i\wedge dv^j\ .
\label{h-rajehto}
\end{equation}
To solve the Euler equations we need \eqref{a-rajehto} and \eqref{h-rajehto} to be constant in time by Lemma \ref{kriteeri}. The crucial observation is that when we fix any $t$ in \eqref{a-rajehto} and \eqref{h-rajehto}, we see that for the time derivatives of $\alpha $ and $h$ to vanish the spatial component only needs to satisfy constraints of the forms 
\begin{equation}
\sum \gamma_{ijk}g_{ijk}=0\textrm{ and }\sum \gamma _{ij}G_{ij}=0\, , 
\label{yl-rajehto}
\end{equation}
where $\gamma_{ijk}$ and $\gamma _{ij}$ are constants. Furthermore, when we substitute these equations back to \eqref{a-rajehto} and \eqref{h-rajehto}, we obtain the constraints that $p_{ijk}$ and $Q_{ij}$ have to satisfy.

There is the following curious connection between $p_{ijk}$ and $Q_{ij}$. 
\begin{lemma}
\[
  \Omega= A^\ast\mathsf{vol}_x \wedge \beta=0\ .
\]
\end{lemma}
\begin{proof}
Let first $m=5$ so that $\Omega=\Omega_{12345}dy^1\wedge dy^2\wedge dy^3\wedge dy^4\wedge dy^5$. Let us denote the rows of $A$ by $\hat A_i$. Then computing we see that
\[
  \Omega_{12345}=
  \sum_{j=1}^3
  \det(\hat A_1,\hat A_2,\hat A_3,\hat A_j',\hat A_j)=0\ .
\]
If $m>5$ then similar formulas are valid for each group of five indices.
\end{proof}

We have now our formulas for $\alpha$ and $h$ but we can still simplify them. First it is important to remember that the domain $D$ is simply some parameter domain which has no physical significance. Hence one can look for the ''simplest'' possible parameter domain. For future reference let us record this observation as 
\begin{lemma} Let  $\psi\,:\,\hat D\to D$ be an arbitrary diffeomorphism  and let $\tilde\varphi^t=\varphi^t\circ\psi$. Then $\tilde\varphi$  provides solutions to the Euler equations via formula \eqref{siirto} if and only if $\varphi$ does.
\label{koordi}
\end{lemma}
\begin{proof}
This follows from standard properties of pullback.
\end{proof}

 In particular, if $g_{123}\ne0$ then we can suppose that $v$ is of the form
 \begin{equation}
     v(z)=\big(z_1,z_2,z_3,f^1(z),\dots,f^{m-3}(z)\big)\, .
\label{v-z}
 \end{equation}
In each case in which we find the general solution to $v$, the indices of $v$ are chosen such that $g_{123}=0$ would imply $\det(d\fii)=0$. Thus we will always write $v$ as in \eqref{v-z}.

We can also write the time component in a more convenient way; recall that by QR decomposition we can always write $A=RB$, where $R\in \mathbb{SO}(3)$ and some $3\times 3$ submatrix $(B_i,B_j,B_k)$ of $B$ is an upper triangular matrix.

To see the effect of the QR decomposition on $h$ we first compute
\[
  \langle A_i',A_j\rangle= 
  \langle R^T R'B_i,B_j\rangle+\langle B_i',B_j\rangle \ .
\]
Using the Euler parameters to represent the rotations we note that
\[
  (R')^TR=2H_{a'}H_a^T=-2H_a H_{a'}^T\,.
\]
This implies that 
\begin{equation}
   Q_{ij}= 4 \langle H_a a', B_i\times B_j\rangle+ 
   \langle B_i',B_j\rangle-\langle B_i,B_j'\rangle\ .
    \label{3x3-1a}
\end{equation}
But now we can in a sense eliminate $a$. Let us set $w=4\,H_a a' $
and $\hat w=(0,w_1,w_2,w_3)$. Then we can write
\[
 4\,K_aa'=\hat w\quad\Longrightarrow\quad
 a'=\tfrac{1}{4}\, K_a^T \hat w=\tfrac{1}{4}\,\tilde K_{\hat w}^T a\ .
\]
Note that $\tilde K_{\hat w}$ is antisymmetric which guarantees that $|a|$ stays constant. $\tilde K_{\hat w}$ is also a Hamiltonian matrix 
so that $a'=\tfrac{1}{4}\,\tilde K_{\hat w}^T a$ is in fact a Hamiltonian system with respect to standard symplectic form
\[
\omega=da_0\wedge da_1+da_2\wedge da_3\ .
\]
 We can use $w$ as a kind of extra  variables in our equations and we can thus write
\begin{equation}
  Q_{ij}=  \langle w, B_i\times B_j\rangle+ 
   \langle B_i',B_j\rangle-\langle B_i,B_j'\rangle\ .
\label{Q-kaava}    
\end{equation}
In what follows we will often have conditions where $p_{ijk}$ or $Q_{ij}$ is some constant; in such cases these constants are always denoted by $p_{ijk}=e_{ijk}$ and $Q_{ij}=c_{ij}$.

Note that the presentation of $\fii=A(t)v(z)$ is not unique.
\begin{lemma}
If $\fii=A(t)v(z)$ is a solution to the Euler equations and $\tilde{v}(z)=H^{-1}v$, where $H$ is a regular $m\times m$ matrix, then there is a matrix $\tilde{A}(t)$, namely $\tilde{A}=AH$, such that $\fii = Av=\tilde{A}\tilde{v}$.
\label{redusointi}
\end{lemma}
Although the statement of the Lemma is really trivial, it is  actually quite useful in concrete computations and we will use it repeatedly to simplify the problems below without losing any generality. 

Then we note the following simple fact which allows us to discard uninteresting cases.
\begin{lemma}
Suppose that in the representation \eqref{yrite} 
\begin{enumerate}
    \item $A_j$ are linearly dependent over $\mathbb{R}$ or 
    \item $dv^j$ are linearly dependent over $\mathbb{R}$.
\end{enumerate}
Then we can write $\varphi=\hat A\hat v$ where $\hat A(t)\in\mathbb{R}^{3\times k}$ and $k<m$.
\label{redusointi2}
\end{lemma}
The simple proof can be found in \cite{toju}. The problems we study typically have several families of solutions. Some of them are reducible to a case of lower $m$ by Lemma \ref{redusointi2} and thus do not need to be considered in the same context. In the solutions given below we have verified that all the reducible solution families have been discarded and that they contain all the relevant solution families.

In the following sections we find several solution classes for cases $m=3$, $m=4$, $m=5$, and $m=6$. For cases $m=4$ and $m=5$ there are subcases that we do not consider here since for each $m$ we concentrate on the solutions for which the spatial component has the most freedom. In case $m=6$ the systems for $A$ will always be overdetermined and we are lucky to find out that there are at least some solution classes in that case, too.

Let us note that the explicit solutions given in \cite{yakzen} are particular cases of either $m=3$ or $m=4$.

\section{Cases $m=3$ and $m=4$}
If $A\in \mathbb{R}^{3\times 3}$ and $v\in \mathbb{R}^3$, any constraint of the form \eqref{yl-rajehto} would imply $\alpha=0$ and thus we may as well assume that $v=z$ by Lemma \ref{koordi}. This case might be called the Kirchhoff case because it is analogous to the Kirchhoff solution in 2D case. 

In this case for \eqref{a-rajehto} and \eqref{h-rajehto} to be independent of time we need $p_{123}=\det(A)$, as well as $Q_{12}$, $Q_{13}$, and $Q_{23}$ to be constant. By scaling we may assume that $A\in \mathbb{SL}(3)$ and by elementary counting arguments we suspect that there are 5 arbitrary functions in the general solution. Writing $A=RB$, where $R\in \mathbb{SO}(3)$ and $B$ is an upper triangular matrix, the condition $\alpha=1$ gives $b_{33}=1/(b_{11}b_{22})$. Then solving $Q_{ij}=c_{ij}$ where $Q_{ij}$ are as in \eqref{Q-kaava} we obtain
\begin{align*}
    b_{12}'=&\frac{b_{12} b_{11}'-w_3 b_{11} b_{22}+c_{12}}{b_{11}}\\[1mm]
      b_{13}'=&
      \frac{b_{13} b_{11}'+w_2 b_{11} b_{33}-w_3 b_{11} b_{23}+c_{13}}{b_{11}}\\[1mm]
        b_{23}'=&
        \frac{b_{11} b_{23} b_{22}'-w_1 +b_{11} c_{23}-b_{12} c_{13}+b_{13} c_{12}}{b_{22} b_{11}}\ .
\end{align*}
Hence there are 5 functions which can be chosen arbitrarily: $b_{11}$, $b_{22}$, $w_1$, $w_2$ and $w_3$. Moreover, to demonstrate the usefulness of Lemma \ref{redusointi} we present the following simplification. 
\begin{lemma}
Without loss of generality $c_{12}$ and $c_{13}$ can be chosen to be zero.
\end{lemma}
\begin{proof}
If some $c_{ij}$ is nonzero, by symmetry we may assume $c_{23}\neq 0$. Then letting $\tilde A=AH$ and $\tilde v=H^{-1}v$, where
\[
H=\begin{pmatrix}
    1 & 0 & 0 \\
    -c_{13}/c_{23} & 1 & 0 \\
    c_{12}/c_{23} & 0 & 1
\end{pmatrix}\, ,
\]
we have $\fii=Av=\tilde A \tilde v$ and $\tilde c_{12}=\tilde c_{13}=0$ for $\tilde A$.
\end{proof}

Note that in this case  $h$ is also constant in space and that the solution in Euler coordinates is $u=A'A^{-1}x$.

Then let $A\in \mathbb{R}^{3\times 4}$. Let us start with the constraint that  $\alpha$ in \eqref{a-rajehto}  must be constant. Using  $v=\big(z,f(z)\big)$ we obtain
\[
\alpha=p_{123}+p_{124}f_{001}-p_{134}f_{010}+p_{234}f_{100}\ .
\]
Note that we cannot require that all $p_{ijk}$ are constant. Since $p_{234}A_1-p_{134}A_2+p_{124}A_3-p_{123}A_4=0$ we must by Lemma \ref{redusointi2} have a constraint of the form
\[
\sum \gamma_{ijk}g_{ijk}=0\ ,
\]
where $\gamma_{ijk}$ are constants.
\begin{lemma}
We may assume this constraint to be $g_{234}=0$.
\end{lemma}
\begin{proof}
We may assume that $\gamma_{234}=1$. Then the transformation $\tilde{v}=Hv$, where
\[
H=\begin{pmatrix}
    1&0&0&0\\
    -\gamma_{134}&1&0&0\\
    \gamma_{124}&0&1&0\\
    -\gamma_{123}&0&0&1
\end{pmatrix},
\]
gives $\tilde{g}_{234}=0$.
\end{proof}

The solution of the equation $g_{234}=0$ is $v=(z_1,z_2,z_3,f(z_2,z_3))$ for an arbitrary function $f$.
Then we must have
\begin{align*}
    p_{123}=& b_{11}b_{22}b_{33}=e_{123}\\
     p_{124}=& b_{11}b_{22}b_{34}=e_{124} \\
      p_{134}=&
       b_{11} \left(b_{23} b_{34}-b_{24} b_{33}\right) =e_{134}
\end{align*}
as well as $Q_{ij}=c_{ij}$ for each $i<j$. We can assume that $e_{123}\ne0$ and by scaling we may further assume that $e_{123}=1$.
\begin{lemma}
Without loss of generality we may assume that  $e_{124}=e_{134}=0$.
\end{lemma}
\begin{proof}
We apply Lemma \ref{redusointi} with
\[
H=\begin{pmatrix}
    1&0&0&0\\
    0&1&0&e_{134}\\
    0&0&1&-e_{124}\\
    0&0&0&1
\end{pmatrix}\ .
\]
\end{proof}

But if $e_{124}=e_{134}=0$ then we have $\alpha=1$ and  we must also have $b_{24}=b_{34}=0$ which gives
\[
B=\begin{pmatrix}
    b_{11}&b_{12}&b_{13}&b_{14} \\
    0&b_{22}&b_{23}&0\\
    0&0&1/(b_{11}b_{22})&0
\end{pmatrix}\ .
\]
This leads to
\begin{theorem}
In case $m=4$ the general solution of the time component can be written in the following form:
\begin{equation}
    A=\hat R\begin{pmatrix}
    1&0&0\\
    0&\cos(\theta)&-\sin(\theta) \\
    0&\sin(\theta)&\cos(\theta)
\end{pmatrix} \begin{pmatrix}
    b_{11}&b_{12}&b_{13}&b_{14}\\
    0&b_{22}&b_{23}&0\\
    0&0&1/(b_{11}b_{22})&0
\end{pmatrix},
    \label{3x4-1ratk}
\end{equation}
where $b_{12}$, $b_{13}$, $b_{14}$ and $b_{23}$ satisfy the equations given in the proof below.  Functions $\theta$, $b_{11}$ and $b_{22}$ are arbitrary  and $\hat R$ is a constant rotation matrix.
\label{m4con1}
\end{theorem}
\begin{proof}
Substituting the above $B$ to the equations $Q_{ij}=c_{ij}$ we first note that $Q_{14}=b_{14} b_{11}'-b_{11} b_{14}'=c_{14}\ne0$ because otherwise the solution reduces. Then we compute that
\[
        b_{12}=\frac{c_{24} b_{11}+c_{12}b_{14}}{c_{14}}\quad,\quad
      b_{13}=
      \frac{c_{34} b_{11}+c_{13}b_{14}}{c_{14}}\ .
\]
Substituting these back to the equations we note that $w_2=w_3=0$ and 
\[
        b_{23}'=
        \frac{b_{11} b_{23} b_{22}'-w_1 }{b_{22} b_{11}}+
       \Big( c_{23}+\frac{c_{12}c_{34}-c_{13}c_{24}}{c_{14}}\Big)\frac{1}{b_{22}}\ .
\]
 Then since $w_1$ is arbitrary we may as well define $\theta'=w_1/2$ in which case one can easily write down the solution of $a'=\tfrac{1}{4}\tilde K_{\hat w}^Ta$. By applying the constant rotation the solution can then be written as indicated.
\end{proof}

Note that if $\varphi$ is a solution  and $\hat R$ is a constant rotation then $\hat R\varphi$ is also a solution. In the sequel we will not anymore write explicitly the rotation in this kind of situation.

The components of vorticity or the Cauchy invariants of the solutions of this form are given by
\[
 h= \Big(c_{24} f_{001}-c_{34} f_{010}+c_{23},  -c_{14} f_{001}-c_{13},c_{14} f_{010}+c_{12} \Big)\ .
\]
Then one could ask how ''general'' vorticity can be constructed with solutions of this form. 
\begin{lemma}
 $\hat h=\big(h^1, h^2 ,h^3\big)$ is a vorticity for some solution given in Theorem \ref{m4con1}, if $h^k$ do not depend on $z_1$, $h^2_{010}+h^3_{001}=0$ and $h^1=e_2h^2+e_3h^3$ for some constants $e_j$.
\end{lemma}
\begin{proof}
Eliminating $f$ from equations $h=\hat h$ gives readily the result.
\end{proof}

Note that since $dh=0$ there is always an equation $h^1_{100}+h^2_{010}+h^3_{001}=0$ for $\hat h$.

If we suppose that $f$ depends only on one variable one could compute other families of solutions. However, these solutions are less interesting physically so we do not analyze them further.

\section{Case $m=5$}

In the case $m=5$ it turns out that there are numerous subcases providing solutions, and it is impossible for us to consider all of them here. Hence, we will concentrate on the cases for which the spatial component has the greatest freedom, which happens when we have two equations of the form $\sum \gamma _{ijk}g_{ijk} = 0$ as the spatial constraints. The systems obtained for the time component are then overdetermined and usually do not give any solutions if the spatial constraints are chosen arbitrarily. However, we found three cases that do provide solutions:
\begin{equation}
	\begin{cases}
		g_{134}+g_{235}=0\\
		g_{135}-g_{234}=0
	\end{cases} \ ,\ 
	\begin{cases}
		g_{134}=0\\
		g_{235}=0
	\end{cases} \ \mathrm{and}\ 
	\begin{cases}
		g_{134}+g_{235}=0\\
		g_{135}=0
	\end{cases} \ .
\label{3x5-tapaukset}
\end{equation}
These are analogous to the three cases we found for the 2D case in \cite[sec. 4]{toju}. As we briefly noted in \cite{toju}, the three 2D cases could be called elliptic, hyperbolic, and parabolic, based on the second-order differential equation that the spatial component functions satisfy in each case. We adopt this terminology here and refer to the three cases of \eqref{3x5-tapaukset} as the \emph{elliptic}, \emph{hyperbolic}, and \emph{parabolic} case, respectively. In the current section we present the solution for the spatial component and the time component for these three cases. As examples of the cases we will not consider any further but which give nontrivial families of solutions  we may cite
\[
	\begin{cases}
		G_{14}+G_{25}=0\\
		G_{15}-G_{24}=0
	\end{cases} ,\, 
	\begin{cases}
		G_{15}=0\\
		G_{24}=0
	\end{cases} ,\, 
	\begin{cases}
		G_{14}+G_{25}=0\\
		G_{15}=0
	\end{cases} .
\]

\subsection{Elliptic case}
Let the spatial constraints be
\[
    g_{134}+g_{235}=
    g_{135}-g_{234}=0 \, .
\]
The solution to this system is $v=(z_1,z_2,z_3,f^1(z_1,z_2,z_3),f^2(z_1,z_2,z_3))$, where $f^1$ and $f^2$ are anti-CR with respect to $z_1$ and $z_2$. For the time component,
\[
    p_{123},p_{345},p_{134}-p_{235},p_{135}+p_{234},p_{124},p_{125},p_{145},p_{245}\, ,
\]
as well as each $Q_{ij}$ has to be constant. Let us write $A=RB$ as before. Since 
\[
 p_{245}A_1-p_{145}A_2+p_{125}A_4-p_{124}A_5=0
\]
by Lemma \ref{plu}, we have $p_{124}=p_{125}=p_{145}=p_{245}=0$ by Lemma \ref{redusointi2}. This immediately implies that $b_{34}=b_{35}=0$.

As for the rest of the determinant constants, we may assume without loss of generality that $p_{123}=p_{345}=1$ and $p_{134}-p_{235}=p_{135}+p_{234}=0$. The proof of this is essentially the same as in the 2D case found in \cite[Lemmas 5.3 and 5.4]{toju}. We simply replace 3 and 4 by 4 and 5 in all the subscripts, respectively, and add a 3 to each of them, which does not affect the calculations. Then we compute that
\[
   \det(d\varphi)=1-\big(f^1_{100}\big)^2-\big(f^1_{010}\big)^2\ne 0\ .
\]
Then we compute that $b_{24}^2+b_{25}^2=b_{22}^2$ and thus we can write $b_{24}=\sin(\theta)b_{22}$ and $b_{25}=\cos(\theta)b_{22}$. But then $b_{14}=\cos(\theta)b_{11}+\sin(\theta)b_{12}$ and $b_{15}=-\sin(\theta)b_{11}+\cos(\theta)b_{12}$ so that 
\[
  B=\begin{pmatrix}
  b_{11}&b_{12}&b_{13}&\cos(\theta)b_{11}+\sin(\theta)b_{12}&-\sin(\theta)b_{11}+\cos(\theta)b_{12}\\
  0&b_{22}&b_{23}&\sin(\theta)b_{22}&\cos(\theta)b_{22}\\
  0&0&1/(b_{11}b_{22})&0&0
  \end{pmatrix}\ .
\]
Then preliminary computations show that in any case we must have
\[
  c_{45}=-c_{12}\ne 0\ , \ 
   c_{24}=c_{15}\ \mathrm{and}\  c_{25}=-c_{14}\ .
\]
\begin{lemma}
Without loss of generality we may assume that $c_{14}=c_{15}=0$.
\label{vakiot}
\end{lemma}
\begin{proof}
Eliminating $w_j$ from the equations $Q_{ij}=c_{ij}$ we eventually find the following system:
\[
  \frac{c_{12}}{b_{11}^{2}+b_{12}^{2}+b_{22}^{2}}
  \begin{pmatrix}
  \cos(\theta)&\sin(\theta)\\
  -\sin(\theta)&\cos(\theta)
  \end{pmatrix}
  \begin{pmatrix}
  2 b_{11} b_{12}\\
  -b_{11}^{2}+b_{12}^{2}+b_{22}^{2}
  \end{pmatrix}=
  \begin{pmatrix}
  c_{14}\\
  c_{15}
  \end{pmatrix}\ .
\]
This implies that $
c_{12}^2-c_{15}^2-c_{14}^2 > 0$,
 and hence  we can choose constants $d$ and $s$ such that
\begin{align*}
    c_{15} &= c_{12}\cos(d)\tanh(2s)\\
    c_{14} &= -c_{12}\sin(d)\tanh(2s)\, .
\end{align*}

Let the transformation matrix $H$ be
\begin{equation}
H=\begin{pmatrix}
\cosh(s) & 0 & 0 & \sinh(s)\cos(d) & \sinh(s)\sin(d) \\
0 & \cosh(s) & 0 & \sinh(s)\sin(d) & -\sinh(s)\cos(d) \\
0 & 0 & 1 & 0 & 0 \\
\sinh(s)\cos(d) & \sinh(s)\sin(d) & 0 & \cosh(s) & 0 \\
\sinh(s)\sin(d) & -\sinh(s)\cos(d) & 0 & 0 & \cosh(s)
\end{pmatrix}\, ,
\label{h-mat-lem61}    
\end{equation}
and let $A=\tilde{A}H$, $\tilde{v}=Hv$. This transformation preserves the spatial constraints as well as the determinant constants we fixed earlier. If $d$ and $s$ are chosen as above, then $\tilde{c}_{14}=\tilde{c}_{15}=0$.
\end{proof}

\begin{theorem}
The solution to the time component is
\[
    \theta'=c_{12}/b_{11}^2=-w_3\, , \, w_1=w_2=0\, , \,  b_{12}=b_{13}=b_{23}=0\, , \,  b_{22}=b_{11}\, .
\]
\end{theorem}
Note that $c_{12}$ must be nonzero.

\begin{proof}
Solving for  $w_j$ and $\theta'$,  and using the previous Lemma, we note that $b_{12}=0$ which then readily yields that $b_{22}=b_{11}$ and
\[
b_{13}=\big(k_0+k_2\cos(\theta)+k_3\sin(\theta)\big)b_{11}\, , \,  b_{23}=\big(k_1-k_3\cos(\theta)+k_2\sin(\theta)\big)b_{11}\, .
\]
By Lemma \ref{redusointi} we can add a multiple of any $A_j$ to $A_3$ without loss of generality and thus achieve $b_{13}=b_{23}=0$. Substituting these to the formulas for $w_j$ and $\theta'$ gives the rest.
\end{proof}

The corresponding vorticity is then 
\[
h=c_{12}\begin{pmatrix}
    -f^1_{100}f^1_{001}-f^1_{010}f^2_{001}\\[1mm] f^1_{100}f^2_{001}-f^1_{010}f^1_{001}\\[1mm] (f^1_{100})^2+(f^1_{010})^2+1
\end{pmatrix}\,.
\]

\subsection{Hyperbolic case}
Here we consider the spatial constraints
\[
g_{134} = g_{235}=0\ ,
\]
which easily lead to $v=\big(z_1,z_2,z_3,f^1(z_1,z_3),f^2(z_2,z_3)\big)$.
Then for the time dependent part each $Q_{ij}$ and the following minors have to be constant:
\[
     p_{123}, p_{124}, p_{125}, p_{135}, p_{145}, p_{234}, p_{245}, p_{345}\, .
\]
Similarly to the previous case, we must have $p_{124}=p_{125}=p_{145}=p_{245}=0$ which again implies that $b_{34}=b_{35}=0$. For the rest of the determinant constants we can assume that $p_{123}=1$, $p_{345}=-1$ and $p_{135}=p_{234}=0$. The proof is again essentially the same as that of the corresponding 2D case \cite[Lemma 5.6]{toju}. From this we obtain
\[
  \det(d\varphi)=1-f^1_{100}f^2_{010}\ne 0\ .
\]
Then putting $b_{15}=\ell b_{11}$ we can now write
\[
  B=\begin{pmatrix}
  b_{11}&b_{12}&b_{13}&b_{12}/\ell&\ell b_{11}\\
  0&b_{22}&b_{23}&b_{22}/\ell&0\\
  0&0&1/(b_{11}b_{22})&0&0
  \end{pmatrix}\ .
\]
This gives
\begin{theorem}
We have $w=0$ and 
\[
    \ell'=-1/b_{11}^2\, , \, b_{22}=\ell b_{11}\, , \, b_{12}=b_{13}=b_{23}=0\, .
\]
\end{theorem}
\begin{proof}
We may assume $c_{15}=1$ by scaling. From the equation $Q_{15}=1$ one immediately gets the equation for $\ell$ and otherwise we solve $w_j$ from equations $Q_{12}=c_{12}$, $Q_{13}=c_{13}$ and $Q_{23}=c_{23}$. Substituting everything back we note that
\[
b_{12}=k_1\ell b_{11}\, , \,  b_{22}=k_2\ell b_{11}\, , \, b_{13}= (k_3\ell+k_4)b_{11}\, , \, b_{23}= (k_5\ell+k_6)b_{11}\, .
\]
Without loss of generality we may add a multiple of $A_5$ to $A_2$ and a multiple of $A_1$ to $A_4$ in order to obtain $b_{12}=0$, after which we add a suitable linear combination of the columns of $A$ to $A_3$ to obtain $b_{13}=b_{23}=0$. By scaling we may also assume $k_2=1$. Finally we note that this form of solutions forces $w=0$.
\end{proof}

In this case the vorticity can be written as

\[
h=\Big(
    -f^1_{001}\,,\, -f^2_{001}\,,\, f^1_{100}+f^2_{010}\Big)\ .
\]

\subsection{Parabolic case}

Here we suppose that 
\[
g_{134}+g_{235}=g_{135}=0\ ,
\]
which implies that $v$ is of the form
\[
v=\big(z, f^1+z_2f^2_{100}, f^2\big)
\]
and $f^j=f^j(z_1,z_3)$.
Then for the time component
\[
p_{123},p_{124},p_{125},p_{145},p_{234},p_{245},p_{345},p_{235}-p_{134}
\]
and each $Q_{ij}$ are constant. Hence by Lemma \ref{plu} and Lemma \ref{redusointi2} we have  $e_{124}=e_{125}=e_{145}=e_{245}=0$. Again, the reduction of the other determinant constants goes as in the corresponding 2D case so by \cite[Lemma 5.8]{toju} we may assume that $e_{234}=1$ and $e_{123}=e_{345}=p_{235}-p_{134}=0$. From this it follows that
\[
  \det(d\varphi)=f_{100}^1+z_2f_{200}^2\ne 0\ .
\]
Thus we have
\[
A_1 = \ell A_2 \ \mathrm{and}\ 
A_5 = \ell A_4\ ,
\]
where $\ell=p_{134}$. Now we have the $B$ matrix in the QR decomposition in the following form
\[
  B=\begin{pmatrix}
  \ell b_{12}&b_{12}&b_{13}&b_{14}&\ell b_{14}\\
  0&0&b_{23}&b_{24}&\ell b_{24}\\
  0&0&0&1/b_{12}b_{23}&\ell/b_{12}b_{23}
  \end{pmatrix}\ .
\]
Substituting to the equations $Q_{ij}=c_{ij}$ we see that $c_{12}$ and $c_{45}$ must be nonzero and that $c_{15}=c_{24}=0$ and $c_{25}=-c_{14}$. Moreover we have
\begin{lemma}
We may assume that $c_{14}=c_{13}=c_{23}=c_{34}=c_{35}=0$.
\end{lemma}
\begin{proof}
Use the transformation $\tilde{A}_2 = A_2 + c_{14}/c_{45}A_4$, $\tilde{A}_1 = A_1 + c_{14}/c_{45}A_5$ to obtain $c_{14}=0$ and thus $c_{25}=0$. After that, write $\tilde{A}_3=A_3+c_{23}/c_{12}A_1-c_{13}/c_{12}A_2-c_{35}/c_{45}A_4+c_{34}/c_{45}A_5$. Since $c_{14}=c_{15}=c_{24}=c_{25}=0$, this transformation gives $c_{13}=c_{23}=c_{34}=c_{35}=0$.
\end{proof}

\begin{theorem}
The solution to the time constraints is
\[
A=\begin{pmatrix}
\ell b_{12} & b_{12} & 0 & 0 & 0\\
0 & 0 & b_{12}^{-2} & 0 & 0\\
0 & 0 & 0 & b_{12} & \ell b_{12}
\end{pmatrix}\, ,
\]
where $\ell'=c_{12}/b_{12}^2$.
\end{theorem}
\begin{proof}
Since so many constants are zero one immediately computes that $\ell'=c_{12}/b_{12}^2$ and $b_{13}=b_{14}=w_2=w_3=0$. Using these we obtain that also $w_1=b_{24}=0$ and then we are left with the equation $c_{45}b_{23}^2b_{12}^4+c_{12}=0$. Scaling gives then the solution above.
\end{proof}

The vorticity is now
\[
  h=c_{12}\begin{pmatrix}
      -f^2_{100}f^2_{001}\\[1mm]
f^1_{100} f^2_{001}-f^2_{100} f^1_{001}+
z_2  \big( f^2_{200} f^2_{001}-f^2_{100} f^2_{101}\big) \\[1mm]
\left(f^2_{100}\right)^{2}+1
  \end{pmatrix} \ .
\]

\section{$m=6$: hyperbolic case}
In these final three sections we study cases where $m=6$. When $m=6$ the computations really become much more involved than for smaller $m$. Also, coming up with spatial constraints that yield solutions to both the spatial and the time component is difficult. We managed to find three sets of spatial constraints that provide solutions, which are again analogous to the three cases we considered in \cite[sec. 4]{toju} and to the ones considered in the previous section. But then we also found that by restricting the time component of these solutions we obtain more general solutions to the spatial component which we could not have come up with otherwise.

In this section we consider the spatial constraints
\[
G_{14} = G_{25} = G_{36} = 0\, .
\]
The solution of the spatial component is
\[
v=(z_1,z_2,z_3,f_1(z_1),f_2(z_2),f_3(z_3))\, .
\]
Incidentally one can check that with this type of $v$ one can find nontrivial solutions in any dimension.

The conditions for the time dependent part give that the following functions are constant:
\begin{align*}
&p_{123}, p_{126}, p_{135}, p_{156}, p_{234}, p_{246}, p_{345}, p_{456}, \\
&Q_{12}, Q_{13}, Q_{15}, Q_{16}, Q_{23}, Q_{24}, Q_{26}, Q_{34}, Q_{35}, Q_{45}, Q_{46}, Q_{56}\, .
\end{align*}
Despite the fact that there are 20 equations to be satisfied by 18 functions, there are nontrivial solutions for the time component; the simplest example is
\[
A=\begin{pmatrix}
e^{c_1t} & 0 & 0 & 0 & 0 & e^{-c_1t} \\
0 & e^{c_2t} & 0 & e^{-c_2t} & 0 & 0 \\
0 & 0 & e^{-(c_1+c_2)t} & 0 & e^{(c_1+c_2)t} & 0
\end{pmatrix}\, ,
\]
where $c_j$ are constants.
\begin{lemma}
We may assume that
\begin{enumerate}
\item $e_{126}=e_{135}=e_{156}=e_{234}=e_{246}=e_{345}=0$\, , or
\item $e_{126}=e_{135}=e_{234}=e_{246}=e_{345}=e_{456}=0$\, .
\end{enumerate}
\label{hypercases}
\end{lemma}

\begin{proof}
Due to symmetry we may assume that $e_{123}\neq 0$. Add a multiple of $A_1$ to $A_4$ to obtain $e_{234}=0$. Now, if also $e_{246}=e_{345}=e_{456}=0$, we obtain case (2) by further adding a multiple of $A_2$ to $A_5$ and a multiple of $A_3$ to $A_6$.

If $e_{246}$, $e_{345}$, and $e_{456}$ are not all zero, we may assume that $e_{456}\neq 0$ by adding a multiple of $A_2$ to $A_5$ or a multiple of $A_3$ to $A_6$ if necessary. In this case we show that we may assume $e_{135}=e_{345}=e_{126}=e_{246}=0$. Apply the transformation $\tilde{A}=AH$, where
\[
H=\begin{pmatrix}
1&0&0&0&0&0 \\
0&1&0&0&c_1&0 \\
0&0&1&0&0&c_2 \\
0&0&0&1&0&0 \\
0&c_3&0&0&1&0 \\
0&0&c_4&0&0&1
\end{pmatrix}\, .
\]
Writing
\[
\hat{H}=\begin{pmatrix}e_{123} & e_{156}\\ e_{234} & e_{456} \end{pmatrix}\, ,
\]
we have
\[
\begin{pmatrix}-\tilde{e}_{135} & \tilde{e}_{126} \\ \tilde{e}_{345} & -\tilde{e}_{246} \end{pmatrix} = \begin{pmatrix}-e_{135} & e_{126} \\ e_{345} & -e_{246} \end{pmatrix} + \hat{H}\begin{pmatrix}c_1 & c_2 \\ c_4 & c_3 \end{pmatrix}\, .
\]
Since $\det(\hat{H})\neq 0$, we may choose the constants $c_j$ such that $\tilde{e}_{135} = \tilde{e}_{126} = \tilde{e}_{345} = \tilde{e}_{246}=0$.

Finally, we may replace $A_1$ and $A_4$ by suitable linear combinations of $A_1$ and $A_4$ to obtain $e_{234}=e_{156}=0$, yielding case (1).
\end{proof}

\begin{lemma}
In case (2) of Lemma \ref{hypercases} the solutions of $A$ reduce to a case of lower $m$.
\end{lemma}
\begin{proof}
We may assume that $e_{123}=e_{156}=1$. To satisfy the determinant conditions we need to have
\begin{align*}
    A_4&=\ell_1A_2+\ell_2A_3\\
    A_5&=\ell_3A_3\quad, \quad
    A_6=(1/\ell_3)A_2\, .
\end{align*}
We may assume $c_{23}=c_{56}=0$. Then the computations show that
\[
A_4 = \frac{c_{46}}{c_{26}}A_2+\frac{c_{45}}{c_{35}}A_3+\frac{c_{34}}{c_{35}}A_5+\frac{c_{24}}{c_{26}}A_6\, .
\]
Since $c_{26}$ and $c_{35}$ have to be nonzero, the columns $A_2$, $A_3$, $A_4$, $A_5$, and $A_6$ are always linearly dependent and the statement follows from Lemma \ref{redusointi2}.
\end{proof}

In case (1) of Lemma \ref{hypercases} we obtain two families of solutions. To describe the computations in detail is not really possible so that we merely present the main ideas which lead to the solutions. However, it is easy to verify that the claimed solutions are actually solutions since one can simply substitute everything to the relevant equations and check that the conditions are satisfied. Of course the computations by hand would be extremely tedious, but with a convenient computer algebra system the verification is easy. 

From the conditions of the case (1) it readily follows that there are functions $\ell_j$ such that we can write either
\begin{align*}
    A_6=&\ell_1A_1,\quad A_4=\ell_2A_2,\quad  A_5=1/(\ell_1\ell_2)A_3
    \quad\mathrm{or}\\
A_5=&\ell_1A_1,\quad  A_6=\ell_2A_2,\quad  A_4=1/(\ell_1\ell_2)A_3\, .
\end{align*}
By symmetry we may choose the first option. We can thus write $A=RB$ where 
\[
B=\begin{pmatrix}
    b_{11}&b_{12}&b_{13}&\ell_2b_{12}&b_{13}/\ell_1\ell_2&\ell_1 b_{11}\\
    0&b_{22}&b_{23}&\ell_2b_{22}&b_{23}/\ell_1\ell_2&0\\
    0&0&1/b_{11}b_{22}&0&1/\ell_1\ell_2b_{11}b_{22}&0
\end{pmatrix}\ .
\]
and 
\[
  \det(d\varphi)=1+f^1_{100}f^2_{010}f^3_{001}\ne 0\ .
\]
Then the computations show that in all cases we must have $w=0$ so there is no rotational component in the solutions and $A=B$.

Then we always have $\ell_1'=-c_{16}/b_{11}^2$ and in particular $c_{16}\ne0$. Using this and other equations we notice that
\[
 (\hat m_2\ell_1^2+\hat m_1\ell_1+\hat m_0) \ell_2+\hat k_1 \ell_1+\hat k_0=0
\]
for some constants $\hat k_j$ and $\hat m_j$. Now we have two cases:
\begin{itemize}
    \item[(i)] it is possible to choose constants $c_{ij}$ such that $\hat k_j=\hat m_j=0$; in this way we avoid a direct algebraic relation between $\ell_1$ and $\ell_2$.
    \item[(ii)] if there is a direct algebraic relation then it turns out that the constants must be chosen so that we can write
    \[
      \ell_2=\frac{1}{m_1\ell_1+m_0}\ .
    \]
\end{itemize}
\begin{theorem}
The case (i) leads to the solution 
\begin{align*}
 &   \ell_1'=-\frac{c_{16}}{b_{11}^2}\quad,\quad
    \ell_2'=-\frac{c_{24}}{b_{22}^2}\\
   &c_{35} \ell_{1}^{2}\ell_{2}^{2} b_{11}^{4} b_{22}^{4} + c_{16}b_{22}^{2} \ell_{2}+ c_{24} b_{11}^{2} \ell_{1}=0\\
   &b_{12}=b_{13}=b_{23}=0\ .
\end{align*}
\label{hyper-m=6}
\end{theorem}
Hence one function, for example $b_{11}$, can be freely chosen.

\begin{proof}
Avoiding algebraic relation between $\ell_1$ and $\ell_2$ forces all constants $c_{ij}$ to be zero except $c_{16}$, $c_{24}$ and $c_{35}$. After this the solution given above follows.
\end{proof}

Note that if we write $\ell_j=\exp(\lambda_j)$ then we obtain
\[
    \lambda_1'\lambda_2'\big(\lambda_1'+\lambda_2'\big)=c_{16}c_{24}c_{35}\ .
\]
Here the vorticity is
\[
  h=-\Big(c_{35}f_2'\,,\,c_{16}f_3'\,,\,c_{24}f_1'\Big)\,.
\]
\begin{theorem}
The case (ii) gives the following solution:
\begin{align*}
    (\ell_1')^3=&k_0\,\ell_1^2(m_1\ell_1+m_0)^2\quad,\quad
    \ell_2=\frac{1}{  m_1\ell_1+m_0}\\
    b_{11}^3=&\frac{k_1}{\ell_1(m_1\ell_1+m_0)}\quad,\quad
    b_{22}=k_2 b_{11}(m_1\ell_1+m_0)\\
     b_{12}=&k_3 b_{11}(m_1\ell_1+m_0)\quad,\quad
        b_{13}=k_4 b_{11}\ell_1\quad,\quad
        b_{23}=k_5 b_{11}\ell_1\ .
\end{align*}
\end{theorem}
\begin{proof}
Starting from $\ell_1'=-c_{16}/b_{11}^2$ and $\ell_2=1/(m_1\ell_1+m_0)$ we find that $ b_{22}=k_2 b_{11}(m_1\ell_1+m_0)$ for some $k_2$. After this the rest follows easily.
\end{proof}

Evidently  $k_0$, $k_1$, $k_2$, $m_0$ and $m_1$ are nonzero. The constants $k_3$, $k_4$ and $k_5$ are arbitrary, except that choosing $k_3=k_4=k_5=0$ is incompatible with the hypothesis $\ell_2=1/(m_1\ell_1+m_0)$. Apparently there is no explicit solution to the differential equation for $\ell_1$. However, there is the following discrete symmetry: if $\mu$ is a solution then $\hat \mu(t)=m_0^2/(m_1^2 \mu(-t))$ is also a solution.

In this case the vorticity has the form
\[
  h=\begin{pmatrix}
      e_0+e_1f_2'+e_2f_3'+e_3f_2'f_3'\\
   e_4+e_5f_1'+e_6f_3'+e_7f_1'f_3'\\
    e_8+e_9f_1'+e_{10}f_2'+e_{11}f_1'f_2'
  \end{pmatrix} \ .
\]
where $e_j$ are some constants.

\subsection{Extensions for the spatial component}

Let us assume that $A$ is of the form
\[
A=\begin{pmatrix}
    b_{11}&0&0&0&0& b_{16}\\
    0& b_{22}&0& b_{24}&0&0\\
    0&0&1/b_{11} b_{22}&0&1/ b_{16} b_{24}&0
\end{pmatrix}\, ,
\]
where the functions $b_{ij}$ are like in Theorem \ref{hyper-m=6}. As we have seen in this case the spatial part is of the form $v=\big(z_1,z_2,z_3,f^1(z_1),f^2(z_2),f^3(z_3)\big)$. However, if we a priori put some restrictions on $b_{ij}$ it is possible to have more general solutions in the spatial part. There are several possibilities to choose these restrictions, for example:
\begin{align*}
 &   \begin{cases}
        b_{11}=b_{22}\\
        b_{16}=b_{24}
    \end{cases}\quad\mathrm{or}\quad
    \frac{b_{11}}{b_{16}}+\frac{b_{16}}{b_{11}}+\frac{b_{22}}{b_{24}}=0\\
 & \mathrm{or}\quad     
 \frac{b_{16}}{b_{11}}+\frac{b_{24}}{b_{22}}+\frac{b_{16}b_{24}}{b_{11}b_{22}}=0 \, .
\end{align*}
All these choices lead to nontrivial solutions but we only consider the first case as an  example to give a flavor of the idea. 

Now if $b_{11}=b_{22}$ and $b_{16}=b_{24}$ then it is easy to compute that actually  $b_{11}=e^{ct}$ and  $b_{16}=e^{-ct}$. It follows that each $Q_{ij}$ is constant so $h$ is automatically independent of time, and for the determinant condition the spatial part only needs to satisfy
\[
g_{145}+g_{256}=g_{236}-g_{134}=g_{125}=g_{346}=0\, .
\]
If $v=(z_1,z_2,z_3,f^1,f^2,f^3)$, then we obtain the following PDE:
\begin{align*}
   & f^2_{001}=
    f^3_{100}+f^1_{010}=0\\
 &   f^1_{100}f^3_{010}-f^1_{010}f^3_{100}=0\\
  &  f^2_{100}f^3_{001}-f^2_{010}f^1_{001}=0\, .
\end{align*}
For simplicity let us assume that
\begin{align*}
    f^1&=g^1(z_1,z_2)q(z_3)\\
     f^2&=g^2(z_1,z_2)\\
      f^3&=g^3(z_1,z_2)q(z_3)\ .
\end{align*}
Substituting this we obtain
\begin{align*}
  &  g^1_{10} g^3_{01}-
    g^1_{01} g^3_{10}=0\\
 &   g^1_{01}+g^3_{10}=0\\
&g^1 g^2_{01}  +g^3g^2_{10} =0\ .
\end{align*}
We can ''solve'' this as follows. From the first equation we have $\nabla g^3=g \, \nabla g^1$ for some $g$. But then the second equation implies that $ g^1_{01}+g\,  g^1_{10}=0$.  But clearly  we must in addition require $ g_{01}+g\,g_{10}=0$ which leads to 
 the following triangular system:
\begin{align*}
   & g_{01}+g\,g_{10}=0\quad,\quad
     g^1_{01}+g\,  g^1_{10}=0\\
   & \nabla g^3=g \, \nabla g^1\quad,\quad
g^1 g^2_{01}  +g^3g^2_{10} =0\,.
\end{align*}
Existence of local solutions now follows from standard theorems. For example  
if $g=(z_1-c_1)/(z_2-c_2)$ then 
\[
g^1=q_1\Big(\frac{z_2-c_2}{z_1-c_1}\Big)\quad
  \mathrm{and}\quad
   g^3=q_2\Big(\frac{z_2-c_2}{z_1-c_1}\Big)
\]
where $q_2'(s)=q_1'(s)/s$.

\section{$m=6$: elliptic case}

Let the spatial constraints be
\[
    G_{34}=
    G_{15}+G_{26}=
    G_{16}-G_{25}=0\, .
\]
Doing the computations we find that $v$ is of the form
\[
v=(z_1,z_2,z_3,f^1(z_3),f^2(z_1,z_2),f^3(z_1,z_2))\, ,
\]
where $f^1$ is arbitrary, and $f^2$ and $f^3$ are anti-CR.

The simplest solution set to the time constraints is
\begin{equation}
A=\begin{pmatrix}
    \cos(\theta_0t) & -\sin(\theta_0t) & \cos(\theta_0t) & \sin(\theta_0t) & 0 & 0\\
    \sin(\theta_0t) & \cos(\theta_0t) & -\sin(\theta_0t) & \cos(\theta_0t) & 0 & 0\\
    0 & 0 & 0 & 0 & \cos(2\theta_0t) & \sin(2\theta_0t)
\end{pmatrix}\ ,
\label{trig-es}
\end{equation}
where $\theta_0\neq 0$ is a constant.

The determinant condition gives the following equations:
\begin{align*}
&p_{123}=e_{123}\ ,\  p_{124}=e_{124},  p_{356}=e_{356}, p_{456}=e_{456}, \\
&p_{135}-p_{236}=e_0, p_{136}+p_{235}=e_1,  p_{145}-p_{246}=e_2 , p_{146}+p_{245}=e_3\, .
\end{align*}
Let us first reduce as many constants to zero as possible.

\begin{lemma}
The determinant constants can be reduced to one of the following three cases:
\begin{enumerate}
    \item $p_{135}-p_{236} = p_{136}+p_{235} = p_{145}-p_{246} = p_{146}+p_{245}=0$.
    \item $p_{356} = p_{456} = p_{136}+p_{235} = p_{146}+p_{245}=0$.
    \item $p_{123} = p_{356} = p_{124} = p_{456} = 0$.
\end{enumerate}
\label{3x6-ell-tapaukset}
\end{lemma}

In the proof of this Lemma we in fact use complex number formulation for conciseness.

\begin{proof}
Let $H$ be of the form
\[
H=\begin{pmatrix}
h_{11} & -h_{21} & 0 & 0 & h_{15} & h_{25} \\
h_{21} & h_{11} & 0 & 0 & h_{25} & -h_{15} \\
0 & 0 & 1 & 0 & 0 & 0 \\
0 & 0 & 0 & 1 & 0 & 0 \\
h_{51} & h_{61} & 0 & 0 & h_{55} & -h_{65} \\
h_{61} & -h_{51} & 0 & 0 & h_{65} & h_{55}
\end{pmatrix}\, ,
\]
and $\tilde{A}=AH$, $v=H\tilde{v}$. $H$ preserves the spatial constraints. Let $h_1=h_{11}+ih_{21}$, $h_2=h_{51}+ih_{61}$, $h_3=h_{15}+ih_{25}$, $h_4=h_{55}+ih_{65}$.

If $p_{123} = p_{356} = p_{124} = p_{456} = 0$, we have case (3). Otherwise we may assume that $p_{123}=1$. Now choosing $h_1=h_4=1$, $h_2=0$, $2h_3=(p_{135}-p_{236})+i(p_{136}+p_{235})$, we may assume that $p_{135}-p_{236}=p_{136}+p_{235}=0$. Then, since
\[
p_{123}p_{356}-p_{135}p_{236}+p_{136}p_{235}=p_{356}-p_{135}^2-p_{136}^2=0\, ,
\]
we have $p_{356}\geq 0$. By scaling $A_5$ and $A_6$ we may assume that $p_{356}=1$ or $p_{356}=0$. By choosing $h_1=1$, $h_2=h_3=0$ and letting $h_4=e^{i\theta}$, where $\theta$ is suitable, we achieve $p_{146}+p_{245}=0$. 

If $p_{356}=0$, then also $p_{135}=p_{136}=p_{235}=p_{236}=0$. It follows that $A_5\times A_6=0$, so $p_{456}=0$. Now the only nonzero constants are $p_{123}$, $p_{124}$, and $p_{145}-p_{246}$, so we have reached case (2).

Then assume that $p_{356}=1$. By transformation $\tilde{A}_4=A_4-p_{456}A_3$ we may assume that $p_{456}=0$. At this point the nonzero constants are $p_{123}=p_{356}=p_{145}-p_{246}=1$ and $p_{124}$. If $|p_{124}|>1$, let $k=p_{124}+\sqrt{p_{124}^2-1}$ and choose $h_1=1$, $h_2=ik$, $h_3=k$, $h_4=i$. After this we have an instance of case (1). If $|p_{124}|=1$, let $h_1=1$, $h_2=0$, $h_3=i$, $h_4=1$, and we obtain case (2). Finally, if $|p_{124}|<1$, let $\sin{\mu}=p_{124}$. Then let $h_1=1$, $h_2=e^{i\mu}$, $h_3=e^{i(\pi-\mu)}$, $h_4=1$. Then $\tilde{A}$ belongs to case (3).
For each transformation it can be checked that $\det(H)\neq 0$.
\end{proof}

Let us inspect the first case.
\begin{lemma}
In case (1) of Lemma \ref{3x6-ell-tapaukset} all solutions for $A$  reduce.
\end{lemma}

\begin{proof}
We may assume that $p_{123}=1$ and $p_{124}=0$. Computations show that we must have $p_{456}=0$ and $p_{356}\neq 0$, and we may assume by scaling that $p_{356}=1$. Then we can write
\begin{align*}
    A_4&=\ell_1A_1+\ell_2A_2\\
    A_5&=\cos{\theta}A_1-\sin{\theta}A_2\\
    A_6&=\sin{\theta}A_1+\cos{\theta}A_2\, .
\end{align*}

Let us write $Q_{15}-Q_{26}=c_1$, $Q_{16}+Q_{25}=c_2$. Similarly to Lemma \ref{vakiot}, we have $(c_{56}-c_{12})^2-c_1^2-c_2^2>0$ so we can write
\begin{align*}
    c_2&= (c_{12}-c_{56})\cos(d)\tanh(2s)\\
    c_1 &= (c_{56}-c_{12})\sin(d)\tanh(2s)\, .
\end{align*}
Thus using 
\[
   H=\begin{pmatrix}
   H_1&0&H_2\\
   0&I&0\\
   H_3&0&H_4
   \end{pmatrix}\ ,
\]
where the blocks $H_j$ are as in the corners of \eqref{h-mat-lem61}, we can 
 bring $c_1$ and $c_2$ to zero. But then we find that
\[
c_{24}A_1-c_{14}A_2+((c_{12}-c_{56})/2)A_4-c_{46}A_5+c_{45}A_6=0\, .
\]
Here $c_{12}-c_{56}=2\theta'b_{11}$ must be nonzero so the columns of $A$ are linearly dependent. Thus the solutions reduce.
\end{proof}

The solutions of the second case also reduce.
\begin{lemma}
In case (2) of Lemma \ref{3x6-ell-tapaukset} all solutions for $A$  reduce to a case of lower $m$.
\end{lemma}
\begin{proof}
Again, we may assume $p_{123}=1$, $p_{124}=0$. If $p_{145}-p_{246}=0$, we can reduce the situation to case (1). Thus we may assume that $p_{145}-p_{246}\neq 0$, and, after adding a multiple of $A_4$ to $A_3$, that $p_{135}-p_{236}=0$.
With these determinant constants we may write
\[
    A_4=p_{234}A_1-p_{134}A_2, \quad A_5=p_{125}A_3, \quad A_6=p_{126}A_3\, .
\]
The conditions for $Q_{35}$ and $Q_{36}$ yield
\[
c_1A_3+c_{36}A_5-c_{35}A_6=0
\]
for some constant $c_1$. Here $c_{35}$ and $c_{36}$ cannot be zero, so this proves there are no irreducible solutions.
\end{proof}

The third case of  Lemma \ref{3x6-ell-tapaukset} does provide solutions. 
 By replacing $A_3$ and $A_4$ by suitable linear combinations of $A_3$ and $A_4$, we can further without loss of generality assume that 
\[
p_{135}-p_{236} = p_{146}+p_{245} = 0, \quad p_{136}+p_{235} = -p_{145}+p_{246} = 1\, ,
\]
and this leads to
\[
\det(d\varphi)=f^2_{100}+f^2_{010}f^1_{001}\ne 0\ .
\]
From this it follows that either $A_1$ and $A_2$ are parallel or $A_5$ and $A_6$ are parallel. Choosing the latter option leads to
\[
B= \begin{pmatrix}
    b_{11} & b_{12} &\ell_1b_{11}+\ell_2b_{12} & -\ell_2b_{11}+\ell_1 b_{12} &  \ell_1b_{15} & -\ell_2 b_{15}\\
    0 & b_{22} & \ell_2 b_{22} & \ell_1 b_{22} & \ell_1b_{25}&  -\ell_2b_{25} \\
    0 & 0 &0 & 0 &\ell_1 b_{35}& -\ell_2 b_{35}
\end{pmatrix}
\]
and in addition 
\[
   b_{11}b_{22}b_{35}\big(\ell_1^2+\ell_2^2\big)+1=0\ .
\]
Then we have
\begin{lemma}
There is a function $\theta$ and constants $k_j$, $\gamma$ such that
\begin{align*}
   & \ell_1=k_1+k_0\cos(\theta), \quad \ell_2=k_2+k_0\sin(\theta)\\
    &  \theta'b_{11}b_{22}=\gamma \ .
\end{align*}
\label{ymp}
\end{lemma}
\begin{proof}
Let $r_{ij}=\langle B_i,B_j\rangle$ as before. Four of the constraints can be written as follows:
\begin{align*}
    Q_{13}&=-r_{11}\ell_1'-r_{12}\ell_2'+c_{12}\ell_2=c_{13}\\
    Q_{14}&=r_{11}\ell_2'-r_{12}\ell_1'+c_{12}\ell_1=c_{14}\\
    Q_{23}&=-r_{12}\ell_1'-r_{22}\ell_2'-c_{12}\ell_1=c_{23}\\
    Q_{24}&=r_{12}\ell_2'-r_{22}\ell_1'+c_{12}\ell_2=c_{24} \ .
\end{align*}
Solving $r_{ij}$ from the first three equations and substituting we obtain from the fourth that $(\ell_1,\ell_2)$ traces a circle. This gives the first formula and we have also $c_{23} = -2 c_{12} k_1+c_{14}$ and $ c_{24} =2 c_{12}k_2 -c_{13}$. Using these $\ell_j$ we can compute the expressions for $r_{ij}$. Since $b_{11}^2b_{22}^2=r_{11}r_{22}-r_{12}^2$ we get the second formula where
\[
  \gamma^2=c_{12}^{2}k_0^2-\big(c_{12} k_1- c_{14}\big)^2 -\big(c_{12} k_2- c_{13}\big)^2 \ .
\]
\end{proof}

Hence to have nontrivial solutions we must have $k_0\ne 0$, $c_{12}\ne 0$ and $\gamma\ne 0$. Without loss of generality we can also assume that $k_1=k\ge 0$, $k_2=0$ and $k_0=1$. It turns out that we have two families of solutions. 
\begin{lemma}
We have $b_{15}=b_{25}=0$, if $k\ne 1$. If $k=1$ then
\begin{align*}
    b_{15} = &
\frac{c_{36}\sin \! \left(\theta \right) -c_{46} (1- \cos \! \left(\theta \right))}{2 \gamma \sin \! \left(\theta \right)}\,b_{22} \\
b_{25} = &
\frac{ m_1\sin \! \left(\theta \right)+m_0 \left(1+  \cos \! \left(\theta \right)\right)}{\gamma^2\left(1+  \cos \! \left(\theta \right)\right) }\, b_{22}\ ,
\end{align*}
where the constants $m_j$ satisfy
\[
\left(c_{36}^{2}+c_{46}^{2}\right) \gamma^{2}+4 c_{12} \big(c_{46} m_{0}+c_{36} m_{1}\big)+4 m_{0}^{2}+4 m_{1}^{2}=0\ .
\]
\label{b1525}
\end{lemma}
Hence both $b_{15}$ and $b_{25}$ must be nonzero, and also $c_{46} m_{0}+c_{36} m_{1}\ne0$.

\begin{proof}
Substituting what we found already to the equations we notice that there are four equations which can be written as 
\[
   M\begin{pmatrix}
       b_{11}\\b_{15}\\b_{25}
   \end{pmatrix}=0\ .
\]
Here the elements of $M$ are polynomials in $\cos(\theta)$, $\sin(\theta)$ and various constants. To have nontrivial solutions all $3\times 3$ minors of $M$ have to be zero. This is possible; a necessary condition in all cases is $c_{35}=c_{45}=0$.  But when $k\ne 1$ then the first column must be zero and the equations do not contain $b_{11}$. However,  the second and third columns are linearly independent for all acceptable choices of constants which implies that $b_{15}=b_{25}=0$.

If $k=1$ then the solution is obtained by simple computation. 
\end{proof}

\begin{theorem}
The solution can be written as follows when $k\ne 1$:
\begin{align*}
  (\theta')^3=&\frac{\gamma^2c_{56}\big(k^2+2k\cos(\theta)+1\big)^2}{k\cos(\theta)+1}\\
  b_{35}=&-\frac{\theta'}{\gamma  \left(k^{2}+2 k \cos \! \left(\theta \right)+1\right)}\\
  b_{11}^2=&\frac{\left(c_{14}-c_{12} k\right) \cos \! \left(\theta \right)+ c_{13}\sin \! \left(\theta \right)-c_{12}}{\theta'}\\
  b_{12}=&\frac{\left(c_{14}-c_{12} k\right)\sin \! \left(\theta \right) -c_{13}\cos \! \left(\theta \right) }{b_{11}\theta'}\\
  b_{22}=&\frac{\gamma}{b_{11}\theta'}\quad,\quad
  w_1=w_2=0\quad,\quad
   w_3=-\frac{\gamma }{b_{11}^2}\ .
\end{align*}
\end{theorem}
Note that to have global solutions we must assume that $k<1$ and that $|c_{12}|$ is sufficiently big.

\begin{proof}
Substituting $b_{15}=b_{25}=0$ to the original equations we see that also $c_{36}=c_{46}=0$.   This  implies that $w_1=w_2=0$ and we also get the formula for $b_{12}$ as above. For $w_3$ we first get
\[
   w_3=\frac{b_{11}b_{12}'-  b_{11}'b_{12}+c_{12}}{b_{11}b_{22}}\ .
\]
This can be computed since $b_{11}b_{12}'-  b_{11}'b_{12}=r_{11}(r_{12}/r_{11})'$ where $r_{ij}$ are as in Lemma \ref{ymp}. 
We have then only three equations involving $\theta$, $b_{11}$ and $b_{35}$ left. Two of them give 
\begin{align*}
&   Q_{56}= b_{35}^{2} \left(k \cos \! \left(\theta \right)+1\right) \theta'=c_{56}\\
& b_{11}b_{22}b_{35}\big(\ell_1^2+\ell_2^2\big)+1=
\frac{\theta'+\gamma  b_{35}\big(k^{2}+2\,k\,  \cos \! \left(\theta \right) +1\big)}{\theta'}=0\ .
\end{align*}
This yields the first two equations and the expression for $b_{11}$ follow from these.
\end{proof}

The vorticity is given by
\[
     h=\begin{pmatrix}
  -c_{13} f^{1}_{001}-2 c_{12} k +c_{14}
\\[1mm]
-c_{14} f^{1}_{001}-c_{13}\\[1mm] 
- c_{56}|\nabla f^2|^2-\left(c_{16}+c_{25}\right) f^{2}_{100}+\left(c_{15}-c_{26}\right) f^{2}_{010}+c_{12}
     \end{pmatrix}\ .
\]

\begin{theorem}
If $k= 1$ the solution can be written as follows: $b_{15}$ and $b_{25}$ are as in Lemma \ref{b1525} and for others we obtain
\begin{align*}
  (\theta')^3=&4\big(c_{56}\gamma^2-c_{36}m_1-c_{46}m_0\big)(1+ \cos(\theta))\\
  b_{35}=&-\frac{\theta'}{2\gamma \left(1+ \cos \! \left(\theta \right)\right)}\\
  b_{11}^2=&\frac{(m_2-m_3)\cos(\theta)-2m_4\sin(\theta)+m_2+m_3}{4\theta'}\\
  b_{12}=& \frac{(m_2-m_3)\sin(\theta)+2m_4\cos(\theta)}{4 b_{11}\theta'} \\
    b_{22}=&\frac{\gamma}{b_{11}\theta'}\quad,\quad w_1 =
 w_2 =0
\quad,\quad w_3= -\frac{\gamma}{b_{11}^2}\ ,
\end{align*}
where
\[
  m_2=\frac{\gamma^{2} c_{36}^{2}+4 m_{0}^{2}}{c_{36} m_{1}+c_{46} m_{0}}\quad,\quad
  m_3=\frac{\gamma^{2} c_{46}^{2}+4 m_{1}^{2}}{c_{36} m_{1}+c_{46} m_{0}}\quad,\quad
  m_4=\frac{\gamma^{2} c_{36}c_{46}-4 m_0m_1}{c_{36} m_{1}+c_{46} m_{0}}\ .
\]
\end{theorem}

\begin{proof}
Substituting $b_{15}$, $b_{25}$ and various constants computed in Lemma \ref{b1525} to our equations we readily get $w_1=w_2=0$, $w_3=-\gamma/b_{11}^2$ and $b_{12}$ as given above. Substituting everything back to equations leaves us with three equations containing $\theta$, $b_{35}$ and $b_{11}$. The expressions for these equations are too big to be written down. But eliminating $b_{11}$ and $b_{35}$ from the equations gives the surprisingly simple equation for $\theta$ given above and then the formulas for $b_{11}$ and $b_{35}$ follow easily. 
\end{proof}

In the solutions $\theta$ reaches $\pi$ in finite time and the solution blows up.

\subsection{Extension for the spatial component}

We can again look for other solutions by restricting the time component and trying to find more general solutions to the spatial component. Let us consider the matrix already mentioned in \eqref{trig-es},
\[
A=\begin{pmatrix}
    \cos(\theta_0t) & -\sin(\theta_0t) & \cos(\theta_0t) & \sin(\theta_0t) & 0 & 0\\
    \sin(\theta_0t) & \cos(\theta_0t) & -\sin(\theta_0t) & \cos(\theta_0t) & 0 & 0\\
    0 & 0 & 0 & 0 & \cos(2\theta_0t) & \sin(2\theta_0t)
\end{pmatrix}\, ,
\]
which is one particular form of the elliptic case. Now the conditions for the spatial component are
\begin{align*}
  &  g_{125}+g_{345}=
    g_{126}+g_{346}=0\\
   & g_{135}-g_{236}+g_{245}+g_{146}=0\\
&    g_{136}+g_{246}-g_{145}+g_{235}=0\, .
\end{align*}
Writing $v=\big(z,f(z)\big)$ as before and supposing that $f^2$ and $f^3$ are not anti CR with respect to $(z_1,z_2)$ we obtain a system of the form 
\[
  f^k_{001}=g^k \quad,\quad 1\le k\le 3
  \quad,\quad f^3_{010}=g^4\ ,
\]
where the (real analytic) functions $g^k$ do not contain derivatives with respect to $z_3$. The analysis of this system turns out to be surprisingly tricky, so we simply give some indications of the necessary steps. The terms in italics are defined and explained in \cite{werner}.

If we had only the equations $ f^k_{001}=g^k$ for  $ 1\le k\le 3$ then the local solution with initial conditions $f^k(z_1,z_2,0)=h^k(z_1,z_2)$ would exist by Cauchy--Kovalevskaia Theorem. Now one could hope that if the functions $h^k$  satisfied the final equation then this would lead to a solution. Indeed there is a generalization of Cauchy--Kovalevskaia Theorem, namely \emph{Cartan--K\"ahler Theorem}, which can be applied to the overdetermined case. However, this theorem requires that the system is \emph{involutive}; in this case one can write the system in the \emph{Cartan normal form} from which one can see what kind of initial conditions are appropriate.

Unfortunately it happens that our system is not involutive. 
Now any system can in principle be transformed to an involutive form by \emph{prolongations} and \emph{projections}. This is sometimes called the \emph{Cartan--Kuranishi algorithm}. Typically the concrete computations are difficult, even for a computer. We checked that the \emph{symbol} of the system becomes involutive after two prolongations, and then there is one \emph{integrability condition} of third order. We did not check if adding this integrability condition produces an involutive system or if there are further integrability conditions of higher order. Anyway in this way one can in principle compute the involutive form after which the existence of the local solution follows from the Cartan--K\"ahler theorem.

\section{$m=6$: parabolic case}

This case has the spatial constraints
\[
    G_{15}+G_{26}=G_{16}= G_{34}=0\, .
\]
The solution to this is
\[
v=(z_1,z_2,z_3,f_1(z_3),z_2f_3'(z_1)+f_2(z_1),f_3(z_1))\, .
\]
For the time component the following quantities must be constant:
\begin{align*}
    &p_{123}, p_{356}, p_{235}, p_{135}-p_{236}, \\
    &p_{124}, p_{456}, p_{245}, p_{145}-p_{246}, \\
    &Q_{12}, Q_{13}, Q_{14}, Q_{15}-Q_{26}, Q_{25}, \\
    &Q_{23}, Q_{24}, Q_{35}, Q_{36}, Q_{45}, Q_{46}, Q_{56}\, .
\end{align*}
The determinant constants can be reduced to several different cases. It is best to omit the consideration of the ones that do not give solutions as they just contain long calculations similar to what we have seen. The only reduced case for the determinant constants that provides irreducible solutions is where
\begin{align*}
    &p_{123}=0, p_{356}=0, p_{235}=1, p_{135}-p_{236}=0, \\
    &p_{124}=0, p_{456}=0, p_{245}=0, p_{145}-p_{246}=1.
\end{align*}
It turns out that we must have either $p_{125}=p_{345}=0$ or $p_{234}=p_{256}=0$ and we may assume $p_{125}=p_{345}=0$ by exchanging $A_1$ with $A_6$ and $A_2$ with $A_5$ if needed. The determinant conditions are thus satisfied if 
\begin{align*}
    A_1&=\ell_1A_2\quad,\quad
    A_4=\ell_2A_5\\
    A_6&=\frac{1}{\ell_2}A_3+\ell_1A_5
\end{align*}
for some functions $\ell_1$ and $\ell_2$ and in this case we have
\[
  \det(d\varphi)=f^2_{100}+z_2f^3_{200}-f^1_{001}f^3_{100}\ne 0\ .
\]
Now using the QR decomposition with respect to columns 2, 3 and 5 we can thus write 
\[
  B= \begin{pmatrix}
   \ell_1 b_{12} & b_{12} & b_{13} & \ell_2b_{15}& b_{15}& \ell_1b_{15}+b_{13}/\ell_2  \\
    0 & 0 & b_{23} & \ell_2b_{25} & b_{25} & \ell_1b_{25}+b_{23}/\ell_2 \\
    0 & 0 & -1/b_{12}b_{25} & 0 &0& -1/(\ell_2 b_{12}b_{25})
    \end{pmatrix}\, .
\]
The analysis then is surprisingly similar to the hyperbolic case in Section 6. Again it is not possible to give the computations in detail, but  the basic idea is as follows. First the preliminary computations show that we must have $w=0$ so that there is no rotational component. Then computing further we obtain
\[
 (\hat m_1\ell_2+\hat m_0) \ell_1+\hat k_2\ell_2^2+\hat k_1 \ell_2+\hat k_0=0
\]
for some constants $\hat k$ and $\hat m$. At first it is possible to choose constants $c_{ij}$ such that $\hat k_j=\hat m_j=0$ and thus apparently to avoid the direct algebraic relation between $\ell_1$ and $\ell_2$. However, the system is such that later we get this type of relation anyway. Moreover it turns out that $\hat m_0=0$ in any case. Hence to get nontrivial solution we must always have
\begin{equation}
   \ell_1=k_2\ell_2+k_1+k_0/\ell_2\ .
\label{yrite-l}    
\end{equation}
We thus get three families of solutions: (1) $k_2k_0\ne0$, (2) $k_0=0$ and (3) $k_2=0$.
\begin{lemma}
Suppose that $k_2k_0\ne0$; then 
\begin{align*}
\left(\ell_{2}'\right)^{3} = &
\frac{k_3\ell_{2}^{4}  }{ k_0-k_2 \ell_{2}^{2}}&
l_{1} = &
k_2\ell_2+k_1+k_0/\ell_2\\
b_{25}^6=&\frac{c_{45}^3( k_0-k_2 \ell_{2}^{2})}{k_3\ell_{2}^{4}}&
    b_{13}=&\,b_{15}=0\\
    b_{23}=&(k_4\ell_2-k_0)b_{25}&
b_{12}^2=&\frac{c_{12} b_{25}^2\ell_2^2}{c_{45}( k_0-k_2 \ell_{2}^{2})}
\end{align*}
where $k_j$ are some constants.
\end{lemma}
\begin{proof}
Starting from the equation \eqref{yrite-l} we first get some differential equations for $b_{13}$ and $b_{15}$. However, substituting these back to the system we notice that in fact  $b_{13}=b_{15}=0$ and $\ell_2'=c_{45}/b_{25}^2$. Using these the rest follows easily.
\end{proof}

Note that the differential equation for $\ell_2$ has a discrete symmetry: if $\mu$ is a solution then also $k_0/(k_2\mu)$ is a solution. Apparently one cannot explicitly solve the equation for $\ell_2$.

The vorticity is now
\[
  h=\begin{pmatrix}
  -f^{3}_{100} \left(c_{45} f^{1}_{001}+c_{35}\right)\\[1mm]
z_2 \left(c_{45} f^{1}_{001}+c_{35}\right) f^{3}_{200}+\left( c_{45}f^{2}_{100}+r_{1}f^{3}_{100} \right) f^{1}_{001}+
c_{35}f^{2}_{100} - r_{0}f^{3}_{100}
\\[1mm] 
c_{45} k_{2} \left(f^{3}_{100}\right)^{2}+\left(c_{15}-c_{26}\right) f^{3}_{100}+c_{12}
  \end{pmatrix}\ ,
\]
where $r_j$ are some constants.

\begin{lemma}
Let us assume that $k_0=0$; then 
\begin{align*}
l_{1} = & k_2k_4t^3+k_1&
l_{2} =& k_4t^3&
b_{12} =&k_5/t&
b_{13} = &k_6t^2\\
b_{15} = &0&
b_{23} = &k_7t^2&
b_{25} = &k_3/t
\end{align*}
where $k_j$ are constants.
\label{2-perhe}
\end{lemma}
\begin{proof}
Substituting $w=0$ and $\ell_1=k_2\ell_1+k_1$ we can solve $b_{13}$, $b_{15}$, $b_{23}$ and $b_{25}$ in terms of $\ell_2$ and $b_{12}$. Substituting this back we obtain a system for $\ell_2$ and $b_{12}$ whose solution is as above. Then the rest follows easily.
\end{proof}

\begin{lemma}
Let us assume that $k_2=0$; then
\begin{align*}
 l_{1} = &k_1+k_0t^3/k_4&
l_{2} =&k_4/t^3 &    b_{15} = &
k_6t^2&  b_{25} =&k_5t^2\\
b_{12} =& k_3/t &
b_{13} =&-k_0k_6t^2+k_7/t
&
b_{23} = &
-k_0k_5t^2+k_8/t
\end{align*}
where $k_j$ are constants.
\label{3-perhe}
\end{lemma}
\begin{proof}
Substituting $w=0$ and $\ell_1=k_1+k_0/\ell_2$ we can solve $b_{12}$, $b_{13}$ and $b_{23}$ in terms of $\ell_2$ and $b_{25}$, and in particular $b_{12}=\hat k b_{25}\ell_2$ for some $\hat k$. Substituting these back to the equations we are left with equations for $\ell_2$ and $b_{25}$ whose solution is given above. After this the rest follows easily.
\end{proof}

Note that $b_{12}$ is the same in both Lemma \ref{2-perhe} and Lemma \ref{3-perhe}.

\subsection{Extension for the spatial component}
In both Lemmas \ref{2-perhe} and \ref{3-perhe} the entries of $A$ are constant linear combinations of $t^2$ and $1/t$, so by linear transformations we can bring $A$ to the form
\[
A=\begin{pmatrix}
    t^2 & 1/t & 0 & 0 & 0 & 0 \\
    0 & 0 & t^2 & 1/t & 0 & 0 \\
    0 & 0 & 0 & 0 & t^2 & 1/t
\end{pmatrix}\, .
\]
Let us present the constraints for $v$ when $A$ is as above. The equations $v$ needs to satisfy are
\[
    g_{135}=g_{246}=g_{235}+g_{136}+g_{145}=0 \, .
\]
Writing $v=(z,f(z))$ this gives
\begin{align*}
    &f^2_{010}=0\\
    &f^1_{100}f^3_{001}-f^3_{100}f^1_{001} = 0\\
    &f^2_{100}-f^3_{010}+f^1_{010}f^2_{001} = 0 \, .
\end{align*}
Since $f^2$ does not depend on $z_2$, the last equation gives
\[
f^3=z_2f^2_{100}+f^1f^2_{001}+g(z_1,z_3)\, .
\]
Substituting this to $f^1_{100}f^3_{001}-f^3_{100}f^1_{001} = 0$ gives a standard first order PDE for $f^1$ so the local solution exists and $f^2$ and $g$ are arbitrary. Note that here $f^1$ and $f^3$ depend nontrivially on $z_2$.

%%%%%%%%%%%%%%%%%%%%%%%%%%%%%%%%%%%%%%%%%%%%%%%%%%%%%%%%%%%%%%%%%%%%%%

%%%%%%%%%%%%%%%%%%%%%%%%%%%%%%%%%%%%%%%%%%%%%%%%%%%%%%%%%%%%%%%%%%%%%%

\end{document}